\documentclass[10pt]{article}
\usepackage[utf8]{inputenc}
\usepackage{tcolorbox}
\usepackage{graphicx}
\usepackage{array}
\usepackage{verbatim}
\usepackage{hyperref}
\usepackage{float}
\usepackage{amsmath}
\usepackage{amsfonts}
\usepackage{dsfont}
\usepackage{xcolor}
\usepackage{amsthm}
\usepackage{bbm}
\theoremstyle{definition}

\newtheorem{theorem}{Theorem}[section]

\theoremstyle{remark}

\newcommand{\diff}[1]{\,\mathrm{d}#1}
\newcommand{\R}{{\mathbb{R}}}

\newcommand{\E}{{\mathbb{E}}}
\newcommand{\N}{{\mathbb{N}}}

\newcommand{\F}{{\mathcal{F}}} 
\renewcommand{\P}{{\mathbb{P}}} 

\title{A randomised trapezoidal quadrature}
\author{
  Yue Wu \footnote{
  Mathematical Institute,
  University of Oxford,
  Oxford,  OX2 6GG, UK; 
  Alan Turing Institute,
  London, NW1 2DB, UK 
  \texttt{yue.wu@maths.ox.ac.uk}}}
  
\begin{document}

\maketitle
\begin{abstract}
  A randomised trapezoidal quadrature rule is proposed for continuous functions which enjoys less regularity than commonly required. Indeed, we consider functions in some fractional Sobolev space. Various error bounds for this randomised rule are established while an error bound for classical trapezoidal quadrature is obtained for comparison. The randomised trapezoidal quadrature rule is shown to improve the order of convergence by half. 
  \vspace{0.2cm}
  
  \vspace{0.2cm}
  
\end{abstract}

\section{Introduction}\label{sec;intro}
It is well known that the trapezoidal quadrature in classical numerical analysis is a technique for approximating $\mathbb{R}^d$-valued definite integral when the integrand is at least twice differentiable. Without loss of generality we consider the time interval is $[0,T]$ and $g\in C^{2}([0,T])$ is the integrand of interest, where $C^{2}([0,T]):=C^{2}([0,T];\mathbb{R}^d)$ is the space of $\mathbb{R}^d$-continuous functions, endowed with the uniform norm topology, that have continuous first two derivatives. The trapezoidal quadrature is proven to achieve order of convergence $2$ for evaluating the integral $I[g]:=\int_0^T g(t)\diff{t}$ with finite many point evaluations \cite{2007integration}. To implement this, partition the interval $[0,T]$ into $N$ equidistant intervals with stepsize $h_N=\frac{T}{N}$, i.e., 
\begin{equation}\label{eqn:partition}
    \Pi_h:=\{t_j:=jh\}_{j=0}^N\subset [0,T],
\end{equation}
where the subscription $N$ is suppressed in h for the
sake of notational simplicity but
assumed implicitly in all of the quantities introduced involving $h$. Define
\begin{equation}\label{eqn:trapezoidal}
   Q_h[g]:=\frac{h}{2}\sum_{i=0}^{N-1}\big(g(t_i)+g(t_{i+1})\big).
\end{equation}
When $g$ has less regularity, trapezoidal quadrature will show a slower convergence and a sharp bound \cite{2002sharp}. To accelerate its convergence when $g$ is 'rougher', we consider a \emph{randomised trapezoidal quadrature}, which is inspired by the randomised version of mid-point Runge-Kunta quadrature rule \cite{2017rk} and stochastic version of trapezoidal quadrature for It\^o integral \cite{2018tra}. In this paper, the $\mathbb{R}^d$-valued target function $g$ is assumed to be in fractional Sobolev space $W^{\sigma,p}(0,T)$ under Sobolev-Slobodeckij norm: 
\begin{equation}
    \|g\|_{W^{\sigma,p}(0,T)}=\big(\int_0^T |g(t)|^p \diff{t}+\int_0^T |\dot{g}(t)|^p \diff{t}+\int_0^T\int_0^T\frac{|\dot{g}(t)-\dot{g}(s)|^p}{|t-s|^{1+(\sigma-1)p}}\diff{t}\diff{s}\big)^{\frac{1}{p}},
\end{equation}
for $\sigma\in (1,2)$ and $p\in [2,\infty)$. We may write $\|g\|_{W^{\sigma,p}(0,T)}$ as $\|g\|_{W^{\sigma,p}}$ for short.
Let us define a randomised trapezoidal quadrature,
\begin{equation}\label{eqn:random trap}
   RQ^{\tau,n}_h[g]:=\frac{h}{2}\sum_{i=0}^{n-1}\big(g(t_i+\tau_ih)+g(t_{i}+\bar{\tau}_ih)\big) \mbox{ for }n\in[N],
\end{equation}
where $\{\tau_i\}_{i=0}^{N-1}$ is a sequence of independent and identically (i.i.d.) uniformly distributed random variables on a probability space $(\Omega,\mathcal{F},\mathbb{P})$, $\bar{\tau}_i:=1-\tau_i$ and $[N]:=\{1,\ldots,N\}$. The main result, Theorem \ref{thm:random_trap}, shows that the convergence rate can be improved to $\mathcal{O}(N^{-\sigma-\frac{1}{2}})$ compared to $\mathcal{O}(N^{-\sigma})$ achieved by classical trapezoidal quadrature (Theorem \ref{thm:general_trap}).

The paper is organized as follows. In Section 2 we present
some prerequisites from probability theory. In Section 3 we state and prove error estimates
for both classical trapezoidal quadrature and randomised trapezoidal quadrature. In addition, we also investigate the error estimate in almost sure sense for randomised trapezoidal quadrature in Theorem \ref{th3:Rie_pathwise}, which is proven still superior to classical one. In the last section, we verify the results through several numerical experiments.

\section{Preliminaries}
This section is devoted to a briefly review on essential probability results for audience who are not familiar with probability theory. Most of the contents are repeated material from Section 2 in \cite{2017rk}. One may refer to \cite{2000measure} for a more detailed introduction.

Recall that a \emph{probability space} $(\Omega,\mathcal{F},\P)$
consists of a measurable space $(\Omega,\mathcal{F})$ endowed with a finite
measure $\P$ satisfying $\P(\Omega) = 1$. A random variable $X \colon \Omega \to \R^d$ is
called \emph{integrable} if $\int_{\Omega}|X(\omega)| \diff{\P(\omega)}<\infty$. 
Then, the \emph{expectation} of $X$ is defined as
$$\E[X]:=\int_{\Omega}X(\omega)\diff{\P(\omega)} = \int_{\R^d} x
\diff{\mu_{X}(x)},$$
where $\mu_X$ is distribution of X on its image space. We write $X \in L^p(\Omega;\R^d)$ with $p \in [1,\infty)$ if 
$\int_{\Omega}|X(\omega)|^p \diff{\P(\omega)}<\infty$, where
$L^p(\Omega;\R^d)$ is a Banach space endowed with the norm
\begin{align*}
  \| X \|_{L^p(\Omega;\R^d)} = \big( \E \big[ |X|^p \big]
  \big)^{\frac{1}{p}} = \Big( \int_{\Omega}|X(\omega)|^p \diff{\P(\omega)}
  \Big)^{\frac{1}{p}}.
\end{align*}
We will write $\|X \|_{L^p(\Omega;\R)}$ as $\| X \|_{L^p(\Omega)}$ for short.

We say that a 
family of $\R^d$-valued random variables $(X_m)_{m \in \N}$ is a discrete time
\emph{stochastic process} if we interpret the index $m$ as a time parameter. An crucial concept in our main proof is \emph{martingales}, which is a special case of discrete time stochastic process with many nice properies. If $(X_m)_{m \in \N}$ is an independent family of integrable random variables
satisfying $\E [X_m] = 0$ for each $m \in \N$, then the stochastic process
defined by the partial sums  
\begin{align*}
  S_n := \sum_{m = 1}^n X_m, \quad n \in \N, 
\end{align*}
is a discrete time martingale. One of the most important inequalities for martingale Burkholder–Davis–Gundy inequality. In this paper we need its discrete time version.

\begin{theorem}[Burkholder–Davis–Gundy inequality]
  \label{th:BDG}
  For each $p \in (1,\infty)$ there exist positive constants $c_p$ and $C_p$
  such that for every discrete time martingale $(X_n)_{n \in \mathbb{N}}$ and
  for every $n \in \mathbb{N}$ we have 
  $$c_p \| [X]_{n}^{{1}/{2}} \|_{L^p(\Omega)}
  \leq \big\| \max_{j\in \{1,\ldots,n\} } |X_j| \big\|_{L^p(\Omega)} 
  \leq C_p \big\| [X]_{n}^{{1}/{2}} \big\|_{L^p(\Omega)},$$
  where $[X]_n = |X_1|^2 + \sum_{k=2}^{n} |X_{k}-X_{k-1}|^2$ denotes the
	\emph{quadratic variation} of $(X_n)_{n \in \mathbb{N}}$ up to $n$.
\end{theorem}
\section{Trapezoidal quadratures for a rougher integrand}
This section investigate the errors from trapezoidal rules for approximating integral of $g \in W^{\sigma,p}$. The error bound from classical trapezoidal rule is obtained in Section \ref{sec:ctr} and the ones from randomised trapezoidal rule is in Section \ref{sec:rtr}.
\subsection{Classical trapezoidal quadrature for $g\in W^{\sigma,p}$}\label{sec:ctr}
\begin{theorem}\label{thm:general_trap} If $g\in W^{\sigma,p}(0,T)$ for $\sigma\geq 1$, then we have
\begin{equation}\label{eqn:main bound1}
    |I[g]-Q_h[g]|\leq CT^{1-\frac{1}{p}}h^{\sigma}\|g\|_{W^{\sigma,p}(0,T)},
\end{equation}
where $C$ is a constant that only depends on $p$.
\end{theorem}
\begin{proof}

To show Eqn. \eqref{eqn:main bound1}, we follow \cite{2018tra} to rewrite 
\begin{equation}\label{eqn:rewrite}
    g(t_i)+g(t_{i+1})=2g(t_{i+\frac{1}{2}})+\int_{t_{i+\frac{1}{2}}}^{t_i}\dot{g}(s)\diff{s}+\int_{t_{i+\frac{1}{2}}}^{t_{i+1}}\dot{g}(s)\diff{s},
\end{equation}
where $t_{i+\frac{1}{2}}:=\frac{1}{2}(t_i+t_{i+1})$. Then the LHS of Eqn. \eqref{eqn:main bound1} can be rewritten as
\begin{equation*}
    I[g]-Q_h[g]=\sum_{i=0}^{N-1}E_1^{i,i+1}+\sum_{i=0}^{N-1}E_2^{i,i+1},
\end{equation*}
where
\begin{equation}\label{eqn:e1}
    E^{i,i+1}_1:=\int_{t_i}^{t_{i+1}}(g(t)-g(t_{i+\frac{1}{2}}))\diff{t}=\frac{1}{2}\int_{t_i}^{t_{i+1}}\int_{t_{i+\frac{1}{2}}}^{t}\dot{g}(r)\diff{r}\diff{t},
\end{equation}
and
\begin{equation}\label{eqn:e2}
    E^{i,i+1}_2:=\frac{1}{2}\int_{t_i}^{t_{i+1}}\big(\int_{t_{i+\frac{1}{2}}}^{t_i}\dot{g}(s)\diff{s}+\int_{t_{i+\frac{1}{2}}}^{t_{i+1}}\dot{g}(s)\diff{s}\big)\diff{t}.
\end{equation}
Regarding $E_1^{i,i+1}$, first note that
\begin{align*}
   &V_i:= \frac{1}{h}(g(t_{i+1})-g(t_{i}))\int_{t_{i}}^{t_{i+1}}(t-t_{i+\frac{1}{2}})\diff{t}\\
   &= \frac{1}{h}\int_{t_{i}}^{t_{i+1}}\int_{t_{i+\frac{1}{2}}}^{t}\int_{t_{i}}^{t_{i+1}}\dot{g}(s)\diff{s}\diff{r}\diff{t}=0.
\end{align*}
Then we can rewrite $E_1^{i,i+1}$ as
\begin{align}\label{eqn:e1 alterntive}
\begin{split}
      &  E_1^{i,i+1}=E_1^{i,i+1}-V_i=\frac{1}{h}\int_{t_i}^{t_{i+1}}\int_{t_i}^{t_{i+1}}\int_{t_{i+\frac{1}{2}}}^{t}\dot{g}(r)\diff{r}\diff{s}\diff{t}-V_i\\
  &=\frac{1}{h}\sum_{i=0}^{N-1}\int_{t_i}^{t_{i+1}}\int_{t_i}^{t_{i+1}}\int_{t_{i+\frac{1}{2}}}^{t}(\dot{g}(r)-\dot{g}(s))\diff{r}\diff{s}\diff{t}.
\end{split}
\end{align}
Thus evaluating $\sum_{i=0}^{N-1}E_1^{i,i+1}$ under $L^p$ norm gives
\begin{align}\label{eqn:intermediate_bound}
\begin{split}
   & \Big|\sum_{i=0}^{N-1}E_1^{i,i+1}\Big|\leq  \sum_{i=0}^{N-1}\int_{t_i}^{t_{i+1}}\int_{t_i}^{t_{i+1}}|\dot{g}(r)-\dot{g}(s)|\diff{r}\diff{s}\\
   &\leq \sum_{i=0}^{N-1}h^{\frac{2}{q}}\Big(\int_{t_i}^{t_{i+1}}\int_{t_i}^{t_{i+1}}|\dot{g}(r)-\dot{g}(s)|^p\diff{r} \diff{s}\Big)^{\frac{1}{p}},
\end{split}
\end{align}
where the second line is deduced by applying H\"older's inequality twice, and $\frac{1}{q}:=1-\frac{1}{p}$.
For the case $\sigma=1$ and any $p\geq 2$, we may directly apply the discrete H\"older's inequality to the last term above:
\begin{align*}
    &\sum_{i=0}^{N-1}h^{\frac{2}{q}}\Big(\int_{t_i}^{t_{i+1}}\int_{t_i}^{t_{i+1}}|\dot{g}(r)-\dot{g}(s)|^p\diff{r} \diff{s}\Big)^{\frac{1}{p}}
    \leq C\sum_{i=0}^{N-1}h^{\frac{2}{q}+\frac{1}{p}}\Big(\int_{t_i}^{t_{i+1}}|\dot{g}(r)|^p\diff{r}\Big)^{\frac{1}{p}}\\
    &\leq Ch \big(\sum_{i=0}^{N-1}h\big)^\frac{1}{q}\Big(\sum_{i=0}^{N-1}\int_{t_i}^{t_{i+1}}|\dot{g}(r)|^p\diff{r}\Big)^{\frac{1}{p}}=ChT^{1-\frac{1}{p}}\|g\|_{W^{1,p}}.
\end{align*}
For the case $\sigma> 1$ and any $p \geq 2$, we may first make use of the definition of $W^{\sigma,p}$ and then apply the discrete H\"older's inequality:
\begin{align*}
&\sum_{i=0}^{N-1}h^{\frac{2}{q}}\Big(\int_{t_i}^{t_{i+1}}\int_{t_i}^{t_{i+1}}|\dot{g}(r)-\dot{g}(s)|^p\diff{r} \diff{s}\Big)^{\frac{1}{p}}\\
 &\leq \sum_{i=0}^{N-1}h^{\frac{2}{q}+\frac{1}{p}+\sigma-1}\Big(\int_{t_i}^{t_{i+1}}\int_{t_i}^{t_{i+1}}\frac{|\dot{g}(r)-\dot{g}(s)|^p}{|r-s|^{1+(\sigma-1)p}}\diff{r}\diff{s}\Big)^{\frac{1}{p}}\\
 &= h^{\sigma}\sum_{i=0}^{N-1}h^{\frac{1}{q}}\Big(\int_{t_i}^{t_{i+1}}\int_{t_i}^{t_{i+1}}\frac{|\dot{g}(r)-\dot{g}(s)|^p}{|r-s|^{1+(\sigma-1)p}}\diff{r}\diff{s}\Big)^{\frac{1}{p}}\leq h^{\sigma}T^{1-\frac{1}{p}}\|g\|_{W^{1,p}}.
\end{align*}
For term $E_2^{i,i+1}$, we can follow a similar argument in \cite{2018tra} to show that
\begin{equation}\label{eqn:e2 alternative}
    E_2^{i,i+1}=\frac{1}{2h}\int_{t_i}^{t_{i+1}}\int_{t_i}^{t_{i+1}}\Big(\int_{t_{i+\frac{1}{2}}}^{t_i}(\dot{g}(s)-\dot{g}(r))\diff{r}+\int_{t_{i+\frac{1}{2}}}^{t_{i+1}}(\dot{g}(s)-\dot{g}(r))\diff{r}\Big)\diff{s}\diff{t}.
\end{equation}
Indeed, note that $$(t_i-t_{i+\frac{1}{2}})+(t_{i+1}-t_{i+\frac{1}{2}})=0,$$
If defining a new process
\begin{align*}
   &P_i:=\frac{t_i-t_{i+\frac{1}{2}}}{2h}\int_{t_i}^{t_{i+1}}\int_{t_i}^{t_{i+1}}\dot{g}(r)\diff{r}\diff{t}+\frac{t_{i+1}-t_{i+\frac{1}{2}}}{2h}\int_{t_i}^{t_{i+1}}\int_{t_i}^{t_{i+1}}\dot{g}(r)\diff{r}\diff{t}, 
\end{align*}
then Eqn. \eqref{eqn:e2 alternative} can be obtained through the fact that
\begin{align*}
     &E_2^{i,i+1} =E_2^{i,i+1}-P_i.
\end{align*}
Thus applying a similar argument as for $\sum_{i=0}^{N-1}E_1^{i,i+1}$, we can show that
\begin{equation}\label{eqn:e2 bound}
\Big|\sum_{i=0}^{N-1}E_2^{i,i+1}\Big|\leq CT^{1-\frac{1}{p}}h^{\sigma}\|g\|_{W^{\sigma,p}[0,T]}.
\end{equation}
Finally, we can conclude that
\begin{align*}
    &\big|I[g]-Q_h[g]\big|\leq \Big|\sum_{i=0}^{N-1}E_1^{i,i+1}\Big|+\Big|\sum_{i=0}^{N-1}E_2^{i,i+1}\Big|\leq CT^{1-\frac{1}{p}}h^{\sigma}\|g\|_{W^{\sigma,p}[0,T]}.
\end{align*}
\end{proof}
For classical trapezoidal quadrature (CTQ), Theorem \ref{thm:general_trap} states that its order of convergence would be the same as the regularity of the integrand. The boundary case is when $g\in W^{1,p}$, then the order is $1$.
\subsection{Randomised trapezoidal rules for $g\in W^{\sigma,p}$}\label{sec:rtr}
For the randomised trapezoidal quadrature \eqref{eqn:random trap}, the proof follows the idea of randomised quadrature given by \cite{2017rk}.
\begin{theorem}\label{thm:random_trap} Define $I^n:=\int_0^{t_n}g(t)\diff{t}$ for $n\in [N]$ for $g \in W^{\sigma,p}$ with $\sigma\geq 1$ and $p\geq 2$. Then $RQ^{\tau,n}_h[g]\in L^p(\Omega;\mathbb{R}^d)$ and is an unbiased estimator of $I^n[g]$, i.e., $\mathbb{E}[RQ_h^{\tau,n}[g]]=I^n[g]$. Moreover, it holds true that 
\begin{equation}\label{eqn:main bound2}
    \big\|I[g]-RQ^{\tau,N}_h[g]\big\|_{L^p(\Omega;\mathbb{R}^d)}\leq C_p|T|^{\frac{p-2}{2p}} h^{\frac{1}{2}+\sigma}\|g\|_{W^{\sigma,p}(0,T)},
\end{equation}
where $C_p$ is a constant that depends only on $p$.
\end{theorem}
\begin{proof}
First due to $g\in W^{\sigma,p}$ we have $\|g\|_{L^p([0,T];\mathbb{R}^d)}<\infty$. Recall that $\tau_i\in \mathcal{U}(0,1)$ for each $i\in [N-1]\cup\{0\}$. Then it follows that
\begin{equation*}
    \frac{h}{2}\big(\|g(t_i+\tau_i h)\|^p_{L^p(\Omega;\mathbb{R}^d)}+\|g(t_i+\bar{\tau}_i h)\|^p_{L^p(\Omega;\mathbb{R}^d)}\big)=\int_{t_i}^{t_{i+1}}|g(t)|^p \diff{t}<\infty.
\end{equation*}
Hence $RQ_h^{\tau,n}[g]\in L^p(\Omega)$ for $n\in [N]$. To show $RQ_h^{\tau,n}[g]$ is unbiased, we need to examine each term in RHS of Eqn.\eqref{eqn:random trap} through spelling out the expectation and changing variable, i.e.,
\begin{equation*}
  \frac{h}{2}\mathbb{E}[g(t_i+\tau_ih)]=  \frac{h}{2}\int_0^1g(t_i+r h)\diff{r}=\frac{1}{2}\int_{t_i}^{t_{i+1}}g(t)\diff{t},
\end{equation*}
and 
\begin{equation*}
  \frac{h}{2}\mathbb{E}[g(t_i+\bar{\tau}_ih)]=\frac{h}{2}\int_0^1g(t_i+(1-r) h)\diff{r}=\frac{1}{2}\int_{t_i}^{t_{i+1}}g(t)\diff{t}.
\end{equation*}
Summing these terms up gives that $RQ^{\tau,n}_h[g]$ is unbiased for $I^n[g]$. Furthermore, if define the error term like
\begin{equation}\label{eqn:error term}
    E^n:=I^n[g]-RQ^{\tau,n}_h[g]=\frac{1}{2}\sum_{i=0}^{n-1}\int_{t_i}^{t_{i+1}}\big(2g(t)-g(t_i+\tau_i h)-g(t_i+\bar{\tau}_i h)\big)\diff{t},
\end{equation}
then each summand is a mean-zero random variable, i.e.,
\begin{equation*}
    \mathbb{E}\big[\int_{t_i}^{t_{i+1}}\big(2g(t)-g(t_i+\tau_i h)-g(t_i+\bar{\tau}_i h)\big)\diff{t}\big]=0.
\end{equation*}
Note that the summands are mutually independent due to the independence of $\{\tau_i\}_{i=0}^{N-1}$. In addition, it is easy to show $E^n\in L^p(\Omega;\mathbb{R}^d)$. Therefore, $E^n$ is a $L^p$-martingale. Then applying the discrete version of the Burkholder$-$Davis$-$Gundy inequality leads to
\begin{align}\label{eqn:bdg}
\begin{split}
  &  \big\|\max_{n}|E^n|\big\|_{L^p(\Omega)}\leq C_p\big\|\mathbf{[}E^n\mathbf{]}_{N}^{\frac{1}{2}}\big\|_{L^p(\Omega)} \\
  &=\frac{C_p}{2}\Big\|\Big(\sum_{i=0}^{N-1}\Big|\int_{t_i}^{t_{i+1}}\big(2g(t)-g(t_i+\tau_i h)-g(t_i+\bar{\tau}_i h)\big)\diff{t}\Big|^2\Big)^{\frac{1}{2}}\Big\|_{L^p(\Omega)} \\
    &\leq C_p\Big\|\Big(\sum_{i=0}^{N-1}\Big|\int_{t_i}^{t_{i+1}}\big(g(t)-g(t_i+\tau_i h))\big)\diff{t}\Big|^2\Big)^{\frac{1}{2}}\Big\|_{L^p(\Omega)} \\
    &+ C_p\Big\|\Big(\sum_{i=0}^{N-1}\Big|\int_{t_i}^{t_{i+1}}\big(g(t)-g(t_i+\bar{\tau}_i h)\big)\diff{t}\Big|^2\Big)^{\frac{1}{2}}\Big\|_{L^p(\Omega)},
\end{split}
\end{align}
where in the second line we substitute the quadratic variation $[E^n]_N$. Due to symmetric property, it is easy to see we only need to handle the first term on the RHS of Eqn.\eqref{eqn:bdg}. Note that
\begin{align}\label{eqn:bdgafter}
\begin{split}
   & C_p\Big\|\Big(\sum_{i=0}^{N-1}\Big|\int_{t_i}^{t_{i+1}}\big(g(t)-g(t_i+\tau_i h))\big)\diff{t}\Big|^2\Big)^{\frac{1}{2}}\Big\|_{L^p(\Omega)} \\
   &=C_p\Big\|\sum_{i=0}^{N-1}\Big|\int_{t_i}^{t_{i+1}}\big(g(t)-g(t_i+\tau_i h))\big)\diff{t}\Big|^2\Big\|_{L^{\frac{p}{2}}(\Omega)}^{\frac{1}{2}}\\
   &\leq C_p\Big(\sum_{i=0}^{N-1}\Big\|\int_{t_i}^{t_{i+1}}\big(g(t)-g(t_i+\tau_i h))\big)\diff{t}\Big\|^2_{L^{p}(\Omega;\mathbb{R}^d)}\Big)^{\frac{1}{2}}.
\end{split}
\end{align}
Then we have that 
\begin{align*}
   & C_p\Big(\sum_{i=0}^{N-1}\Big\|\int_{t_i}^{t_{i+1}}\big(g(t)-g(t_i+\tau_i h))\big)\diff{t}\Big\|^2_{L^{p}(\Omega;\mathbb{R}^d)}\Big)^{\frac{1}{2}}\\
   &= C_p\Big(\sum_{i=0}^{N-1}\Big\|\int_{t_i}^{t_{i+1}}\int_{t_{i}+\tau_ih}^t \dot{g}(s)\diff{s}\diff{t}\Big\|^2_{L^{p}(\Omega;\mathbb{R}^d)}\Big)^{\frac{1}{2}}\\
   &\leq C_p\Big(\sum_{i=0}^{N-1}\Big\|\int_{t_i}^{t_{i+1}}\int_{t_{i}+\tau_ih}^t |\dot{g}(s)|\diff{s}\diff{t}\Big\|^2_{L^{p}(\Omega)}\Big)^{\frac{1}{2}}\\
    &\leq C_ph\Big(\sum_{i=0}^{N-1}\Big|\int_{t_{i}}^{t_{i+1}} |\dot{g}(s)|\diff{s}\Big|^2\Big)^{\frac{1}{2}}\leq C_ph\Big(\sum_{i=0}^{N-1}h^{\frac{2}{q}}\Big|\int_{t_{i}}^{t_{i+1}} |\dot{g}(s)|^p\diff{s}\Big|^{\frac{2}{p}}\Big)^{\frac{1}{2}}.
    \end{align*}
When $p=2$, the term on the right hand side above can be directly bounded by
\begin{equation}\label{eqn:bound_p2}
    C_ph\Big(\sum_{i=0}^{N-1}h^{\frac{2}{q}}\Big|\int_{t_{i}}^{t_{i+1}} |\dot{g}(s)|^p\diff{s}\Big|^{\frac{2}{p}}\Big)^{\frac{1}{2}}\leq C_ph^{\frac{3}{2}}\|g\|_{W^{1,p}}.
\end{equation}
where $\frac{1}{q}+\frac{1}{p}=1$. For $p>2$, we may apply discrete H\"older inequality and get
\begin{align}\label{eqn:bound_plarger2}
\begin{split}
        &C_ph\Big(\sum_{i=0}^{N-1}h^{\frac{2}{q}}\Big|\int_{t_{i}}^{t_{i+1}} |\dot{g}(s)|^p\diff{s}\Big|^{\frac{2}{p}}\Big)^{\frac{1}{2}}\\
     &\leq C_ph \Big(\sum_{i=0}^{N-1}h^{\frac{2}{q}\frac{p}{p-2}}\Big)^{\frac{p-2}{2p}}\Big(\sum_{i=0}^{N-1}\int_{t_{i}}^{t_{i+1}} |\dot{g}(s)|^p\diff{s}\Big)^{\frac{1}{p}}\\
     &\leq C_ph^{1+\big(\frac{2p}{q(p-2)}-1\big)\frac{p-2}{2p}}|T|^{\frac{p-2}{2p}}\|g\|_{W^{1,p}}=C_ph^{\frac{3}{2}}|T|^{\frac{p-2}{2p}}\|g\|_{W^{1,p}}. 
\end{split}
\end{align}
Now we have shown Bound \eqref{eqn:main bound2} when $\sigma=1$. For Bound \eqref{eqn:main bound2} under $\sigma>1$, we first note that Eqn.\eqref{eqn:rewrite} remains true if replacing $t_i$ by $t_i+\tau_i h$ and $t_{i+1}$ by $t_i+\bar{\tau}_i h$, i.e.,
\begin{equation}\label{eqn:rewrite1}
    g(t_i+\tau_i h)+g(t_i+\bar{\tau}_i h)=2g(t_{i+\frac{1}{2}})+\int_{t_{i+\frac{1}{2}}}^{t_i+\tau_i h}\dot{g}(s)\diff{s}+\int_{t_{i+\frac{1}{2}}}^{t_i+\bar{\tau}_i h}\dot{g}(s)\diff{s}.
\end{equation}
 Thus the second line of Eqn.\eqref{eqn:bdg} can be further splitted as the follows:
\begin{align*}
   & \frac{C_p}{2}\Big\|\Big(\sum_{i=0}^{N-1}\Big|\int_{t_i}^{t_{i+1}}\big(2g(t)-g(t_i+\tau_i h)-g(t_i+\bar{\tau}_i h)\big)\diff{t}\Big|^2\Big)^{\frac{1}{2}}\Big\|_{L^p(\Omega)}\\
   &\leq C_p\Big\|\Big(\sum_{i=0}^{N-1}\Big|\int_{t_i}^{t_{i+1}}\big(g(t)-g(t_{i+\frac{1}{2}}) \big)\diff{t}\Big|^2\Big)^{\frac{1}{2}}\Big\|_{L^p(\Omega)}\\
   &+ C_p\Big\|\Big(\sum_{i=0}^{N-1}\Big|\int_{t_i}^{t_{i+1}}\int_{t_{i+\frac{1}{2}}}^{t_i+\tau_i h}\dot{g}(s)\diff{s}\diff{t}+\int_{t_i}^{t_{i+1}}\int_{t_{i+\frac{1}{2}}}^{t_i+\bar{\tau}_ih}\dot{g}(s)\diff{s}\diff{t}\Big|^2\Big)^{\frac{1}{2}}\Big\|_{L^p(\Omega)}.
\end{align*}
Similar as in the proof of Theorem \ref{thm:general_trap}, we introduce $E_1^{i,i+1}$ defined in Eqn. \eqref{eqn:e1} and 
\begin{equation}\label{eqn:e2_random}
    E^{i,i+1}_2(\tau):=\frac{1}{2}\int_{t_i}^{t_{i+1}}\big(\int_{t_{i+\frac{1}{2}}}^{t_i+\tau h}\dot{g}(s)\diff{s}+\int_{t_{i+\frac{1}{2}}}^{t_{i}+\bar{\tau}h}\dot{g}(s)\diff{s}\big)\diff{t}.
\end{equation}
As in the proof of Theorem \ref{thm:general_trap}, $E^{i,i+1}_1$ can be handled through the equivalent form Eqn.\eqref{eqn:e1 alterntive} and $E_2^{i,i+1}(\tau)$ can be treated in a similar way as Eqn.\eqref{eqn:e2 alternative} by replacing $t_i$ by $t_i+\tau_i h$ and $t_{i+1}$ by $t_i+\bar{\tau}_i h$ in the inner integral of Eqn.\eqref{eqn:e2 alternative}, i.e.,
\begin{equation}\label{eqn:e2 alternative1}
    E_2^{i,i+1}=\frac{1}{2h}\int_{t_i}^{t_{i+1}}\int_{t_i}^{t_{i+1}}\Big(\int_{t_{i+\frac{1}{2}}}^{t_i+\tau_i h}(\dot{g}(s)-\dot{g}(r))\diff{r}+\int_{t_{i+\frac{1}{2}}}^{t_i+\bar{\tau}_i h}(\dot{g}(s)-\dot{g}(r))\diff{r}\Big)\diff{s}\diff{t}.
\end{equation}
Thus
\begin{align*}
   & \frac{C_p}{2}\Big\|\Big(\sum_{i=0}^{N-1}\Big|\int_{t_i}^{t_{i+1}}\big(2g(t)-g(t_i+\tau_i h)-g(t_i+\bar{\tau}_i h)\big)\diff{t}\Big|^2\Big)^{\frac{1}{2}}\Big\|_{L^p(\Omega)}\\
   &\leq C_p  \Big\|\Big(\sum_i^{N-1}|E_1^{i,i+1}|^2\Big)^{\frac{1}{2}}\Big\|_{L^p(\Omega)}+C_p  \Big\|\Big(\sum_i^{N-1}|E_2^{i,i+1}|^2\Big)^{\frac{1}{2}}\Big\|_{L^p(\Omega)}\\
   &=\frac{C_p}{h}\Big\|\Big(\sum_{i=0}^{N-1}\Big|\int_{t_i}^{t_{i+1}}\int_{t_i}^{t_{i+1}}\int_{t_{i+\frac{1}{2}}}^{t}(\dot{g}(r)-\dot{g}(s))\diff{r}\diff{s}\diff{t}\Big|^2\Big)^{\frac{1}{2}}\Big\|_{L^p(\Omega)}\\
    &+\frac{C_p}{2h}\Big\|\Big(\sum_{i=0}^{N-1}\Big|\int_{t_i}^{t_{i+1}}\int_{t_i}^{t_{i+1}}\Big(\int_{t_{i+\frac{1}{2}}}^{t_i+\tau_i h}(\dot{g}(s)-\dot{g}(r))\diff{r}\\
    &+\int_{t_{i+\frac{1}{2}}}^{t_i+\bar{\tau}_i h}(\dot{g}(s)-\dot{g}(r))\diff{r}\Big)\diff{s}\diff{t}\Big|^2\Big)^{\frac{1}{2}}\Big\|_{L^p(\Omega)},
\end{align*}
where the first term on the right hand side from Eqn.\eqref{eqn:e1 alterntive} and the second term is due to Eqn. \eqref{eqn:e2 alternative1}. Let us now deal with the first term, the second term can be handled in the same way. Following a similar argument in \eqref{eqn:bdgafter}, we have that 
\begin{align*}
    &\frac{C_p}{h}\Big\|\Big(\sum_{i=0}^{N-1}\Big|\int_{t_i}^{t_{i+1}}\int_{t_i}^{t_{i+1}}\int_{t_{i+\frac{1}{2}}}^{t}(\dot{g}(r)-\dot{g}(s))\diff{r}\diff{s}\diff{t}\Big|^2\Big)^{\frac{1}{2}}\Big\|_{L^p(\Omega)}\\
    &\leq \frac{C_p}{h}\Big(\sum_{i=0}^{N-1}\Big\|\int_{t_i}^{t_{i+1}}\int_{t_i}^{t_{i+1}}\int_{t_{i+\frac{1}{2}}}^{t}(\dot{g}(r)-\dot{g}(s))\diff{r}\diff{s}\diff{t}\Big\|^2_{L^{p}(\Omega;\mathbb{R}^d)}\Big)^{\frac{1}{2}}\\
    &\leq C_p\Big(\sum_{i=0}^{N-1}\Big|\int_{t_i}^{t_{i+1}}\int_{t_{i+\frac{1}{2}}}^{t}|\dot{g}(r)-\dot{g}(s)|\diff{r}\diff{s}\Big|^2\Big)^{\frac{1}{2}}\\
    &\leq C_p\Big(\sum_{i=0}^{N-1}h^{\frac{2}{q}}\Big(\int_{t_i}^{t_{i+1}}\big(\int_{t_{i+\frac{1}{2}}}^{t}|\dot{g}(r)-\dot{g}(s)|\diff{r}\big)^p\diff{s}\Big)^{\frac{2}{p}}\Big)^{\frac{1}{2}}\\
    &\leq C_p\Big(\sum_{i=0}^{N-1}h^{\frac{4}{q}}\Big(\int_{t_i}^{t_{i+1}}\int_{t_{i+\frac{1}{2}}}^{t}|\dot{g}(r)-\dot{g}(s)|^p\diff{r}\diff{s}\Big)^{\frac{2}{p}}\Big)^{\frac{1}{2}}\\
        &\leq C_p\Big(\sum_{i=0}^{N-1}h^{\frac{4}{q}+\frac{2}{p}+2(\sigma-1)}\Big(\int_{t_i}^{t_{i+1}}\int_{t_{i+\frac{1}{2}}}^{t}\frac{|\dot{g}(r)-\dot{g}(s)|^p}{|r-s|^{1+(\sigma-1)p}}\diff{r}\diff{s}\Big)^{\frac{2}{p}}\Big)^{\frac{1}{2}}\\
                &= C_ph^{\sigma}\Big(\sum_{i=0}^{N-1}h^{\frac{2}{q}}\Big(\int_{t_i}^{t_{i+1}}\int_{t_{i+\frac{1}{2}}}^{t}\frac{|\dot{g}(r)-\dot{g}(s)|^p}{|r-s|^{1+(\sigma-1)p}}\diff{r}\diff{s}\Big)^{\frac{2}{p}}\Big)^{\frac{1}{2}},
\end{align*}
where we apply H\"older's inequality in Line 4 and 5. Similarly as in \eqref{eqn:bound_p2}, for $p=2$ we have that
\begin{equation*}
    C_ph^\sigma\Big(\sum_{i=0}^{N-1}h^{\frac{2}{q}}\Big(\int_{t_i}^{t_{i+1}}\int_{t_{i+\frac{1}{2}}}^{t}\frac{|\dot{g}(r)-\dot{g}(s)|^p}{|r-s|^{1+(\sigma-1)p}}\diff{r}\diff{s}\Big)^{\frac{2}{p}}\Big)^{\frac{1}{2}}\leq C_p h^{\sigma+\frac{1}{2}}\|g\|_{W^{\sigma,p}}.
\end{equation*}
Applying discrete H\"older inequality for $p>2$ as in \eqref{eqn:bound_plarger2}, we have that
\begin{align*}
    & C_p h^\sigma \Big(\sum_{i=0}^{N-1}h^{\frac{2}{q}}\Big(\int_{t_i}^{t_{i+1}}\int_{t_{i+\frac{1}{2}}}^{t}\frac{|\dot{g}(r)-\dot{g}(s)|^p}{|r-s|^{1+(\sigma-1)p}}\diff{r}\diff{s}\Big)^{\frac{2}{p}}\Big)^{\frac{1}{2}}\\
     &\leq C_p h^{\sigma+\frac{1}{2}}T^{\frac{p-2}{2p}}\|g\|_{W^{\sigma,p}}.   
\end{align*}
Altogether we have achieved Bound \eqref{eqn:main bound2}.
\end{proof}
Compared to Theorem \ref{thm:general_trap}, for fixed integrand, the randomised quadrature rule (RTQ) improves the order of convergence by $\frac{1}{2}$ through incorporating randomness. One may also be interested in investigating the almost sure convergence of RTQ. Indeed, the argument from Theorem 3.2 \cite{2017rk} can be directly adapted here:

\begin{theorem}[Almost sure convergence]
  \label{th3:Rie_pathwise}
Assume that conditions from Theorem \ref{thm:random_trap} are satisfied. 
  Let $(h_m)_{m \in \N} \subset (0,1)$ be an arbitrary sequence of step sizes
  with $\sum_{m = 1}^\infty h_m < \infty$. Then, there exist a nonnegative
  random variable $m_0 \colon \Omega \to \N \cup \{0\}$ and a measurable set
  $A \in \F$ with $\P(A) = 1$ such that for all $\omega \in A$ and $m \ge
  m_0(\omega)$, then for every $\epsilon \in (0,\frac{1}{2})$
  there exist a nonnegative random variable $m_0^\epsilon \colon \Omega \to
  \N_0$ and a measurable set $A_\epsilon \in \F$ with $\P(A_\epsilon) = 1$ such
  that 
  such that for all $\omega \in A$ and $m \ge
  m_0(\omega)$ we have
  \begin{align}
    \label{eq4:errRie4}
    \max_{n \in \{0,1,\ldots,N_{h_m}\}} \Big| I^n[g] -
    RQ^{\tau,n}_{h_m}[g](\omega) \Big| \le h_m^{\frac{1}{2} + \gamma - \epsilon},
  \end{align}
where $N_{h_m}:=\lfloor \frac{T}{h_m}\rfloor$, i.e., the integer part of $\frac{T}{h_m}$.
\end{theorem}

Theorem \ref{th3:Rie_pathwise} ensures that RTQ can achieve a slightly better order of pathwise convergence in almost sure sense compared to CTQ when stepsize is adequately small.
\subsection{Numerical experiments} 
In this section we assess the proposed scheme via different experiments. For simplicity, we fix T=1.
\subsubsection{Example 1}
Consider the function:
\begin{equation}\label{eqn:1d_eg}
    g_\gamma(t):=t^\gamma,
\end{equation}
where $\gamma\in\{\frac{5}{4},\frac{3}{2},\frac{7}{4}\}$, $ g_\gamma\in W^{\sigma,2}(0,T)$, for all $\epsilon\in (1,\frac{1}{2}+\gamma)$ (Sobolev's inequality in \cite{2003sobolev}). The curves of $g_\gamma$ with different values in $\gamma$ can be found in Figure \ref{fig:g_value}.
  \begin{figure}[h!]
  \centering
      \includegraphics[width=0.5\textwidth]{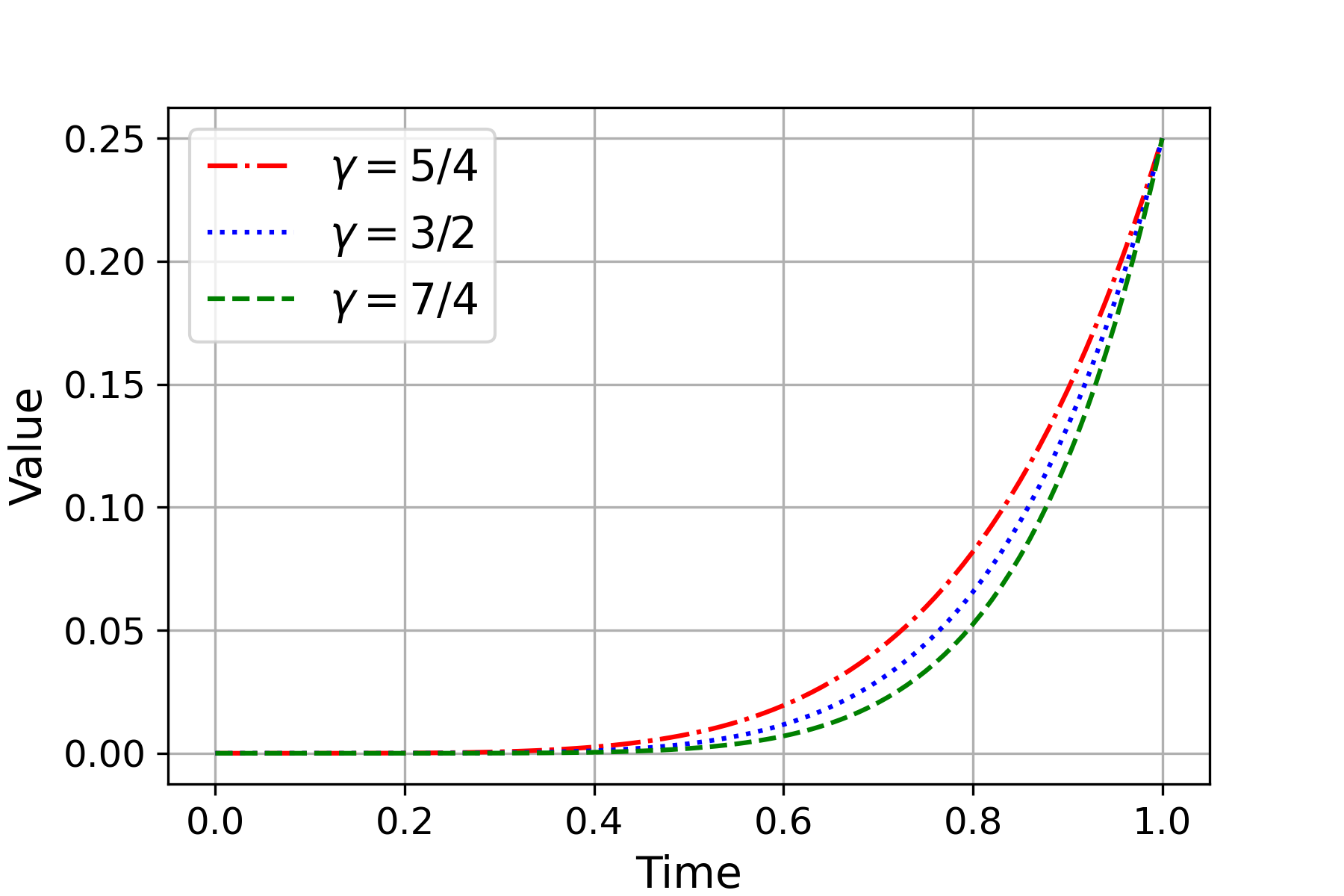}   
  \caption{Function values for $g_\gamma$ under different choices for $\gamma$.  \label{fig:g_value}}
      \end{figure}
The true solution can be easily obtained as $\frac{1}{\gamma+1}$. The
numerical approximations were calculated for both kinds of trapezoidal quadrature with step sizes $h \in \{2^{-i}: i= 5,\ldots, 10\}$ and then compared to the true solution for errors. For RTQ, we computed errors in $L^2$ norm via Monte Carlo method and also computed pathwise error, i.e.,error from one realisation. 

The results of our simulations are shown in Figure \ref{fig:g1_gamma_error_v2} and Table \ref{tab:TQ_order_1d_2v}. Across all different values of $\gamma$, RTQ gave the higher order of convergence compared to CTQ. When $\gamma$ increases from $\frac{5}{4}$ to $\frac{7}{4}$, the order of convergence for RTQ increases eventually to a number very close to $2.5$. Note that the order of convergence for CTQ are not beyond $2$ for all $\gamma$ values. All the performances are superior to theoretical order of convergences shown in Theorem \ref{thm:general_trap} and Theorem \ref{thm:random_trap}. We also examined the computational efficiency of both methods (lower right in Figure \ref{fig:g1_gamma_error_v2}). Though incorporating randomness increases computational expense, RTQ quickly offsets its cost with its higher accuracy.
  \begin{figure}[h!]

 \includegraphics[width=0.49\textwidth]{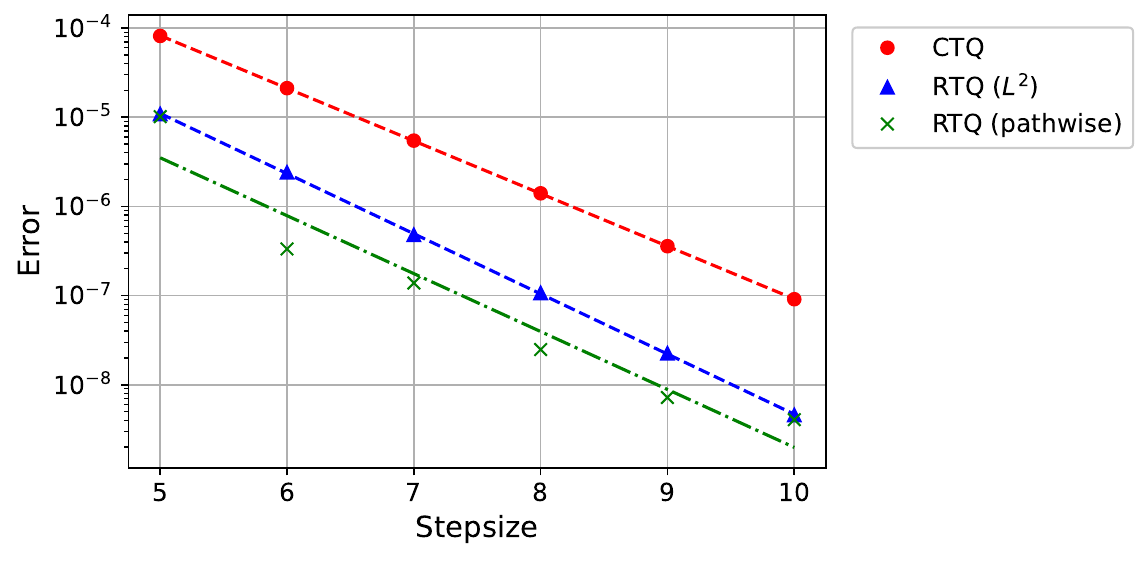}       \includegraphics[width=0.49\textwidth]{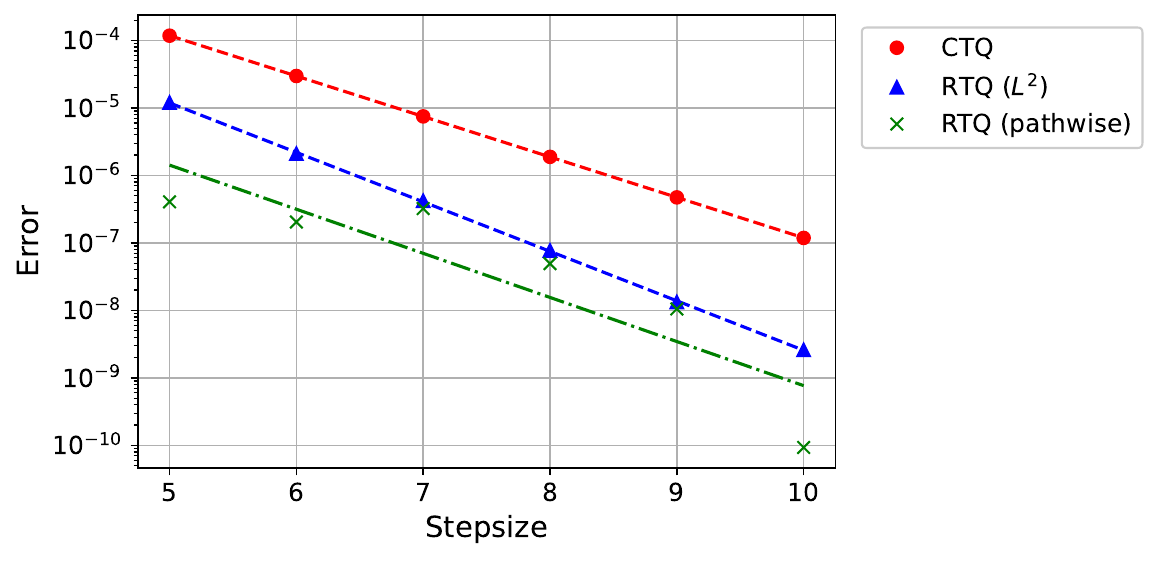} \\   \includegraphics[width=0.49\textwidth]{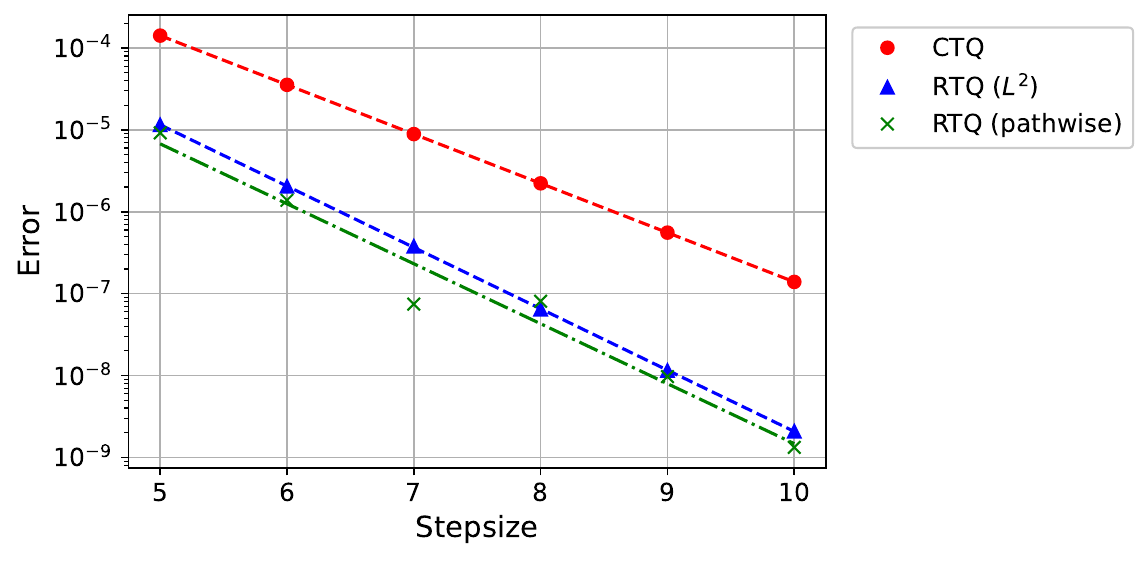}
 \includegraphics[width=0.36\textwidth]{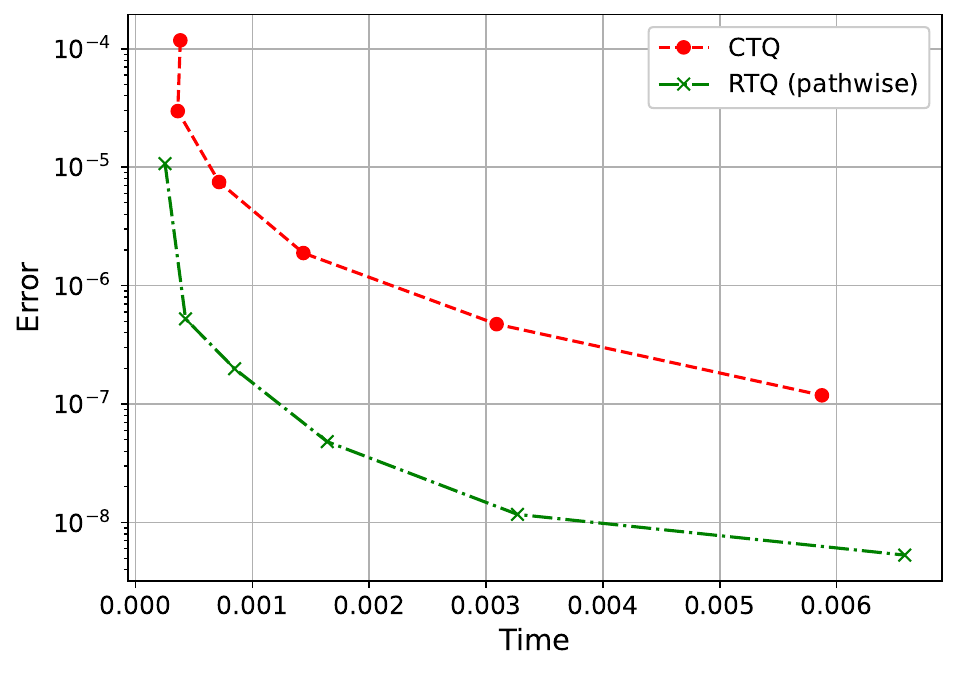}
  \caption{Error plots for approximating $I[g_\gamma]$ via variants of trapezoidal rule under different choices for $\gamma$ (upper left: $\gamma=\frac{5}{4}$; upper right: $\gamma=\frac{3}{2}$; lower left: $\gamma=\frac{7}{4}$) and time cost plot for $\gamma=\frac{3}{2}$ (lower right).  \label{fig:g1_gamma_error_v2}}
      \end{figure}

\begin{table}[ht] 
\centering
\caption{Order of convergences for simulating $I[g_\gamma]$.}
\label{tab:TQ_order_1d_2v}
\vspace{0.2cm}
\begin{tabular}{c|ccc}
\hline
\vspace{0.1cm}
$\gamma$ &  CTQ  &  RTQ ($L^2$)& RTQ (pathwise)\\
\hline
\vspace{0.1cm}
$\frac{5}{4}$ &1.96 &  2.24 &  2.13\\
\vspace{0.1cm}
$\frac{3}{2}$ & 1.99 & 2.44 & 2.17 \\
$\frac{7}{4}$ &1.99&  2.50 &  2.43\\
\hline
\end{tabular}
\end{table}

\subsubsection{Example 2}
Consider the function:
\begin{equation}\label{eqn:gB}
    g_B(t):=\int_0^t B(s) \diff{s},\ \mbox{for }t\in[0,T],
\end{equation}
where $B(s)$ is a realisation of standard Brownian motion (BM) (c.f. Section 3.1 in \cite{2013mc}). It is well known that $B\in C^{\frac{1}{2}-\epsilon}$ for arbitrary small $\epsilon>0$, therefore $g_B\in W^{\frac{3}{2}-\epsilon,p}$ for $p>1$.  Figure \ref{fig:gW_values} illustrates how one BM path looks like and the curve of its $g_B$.
  \begin{figure}[h!]
  \centering
        \includegraphics[width=0.48\textwidth]{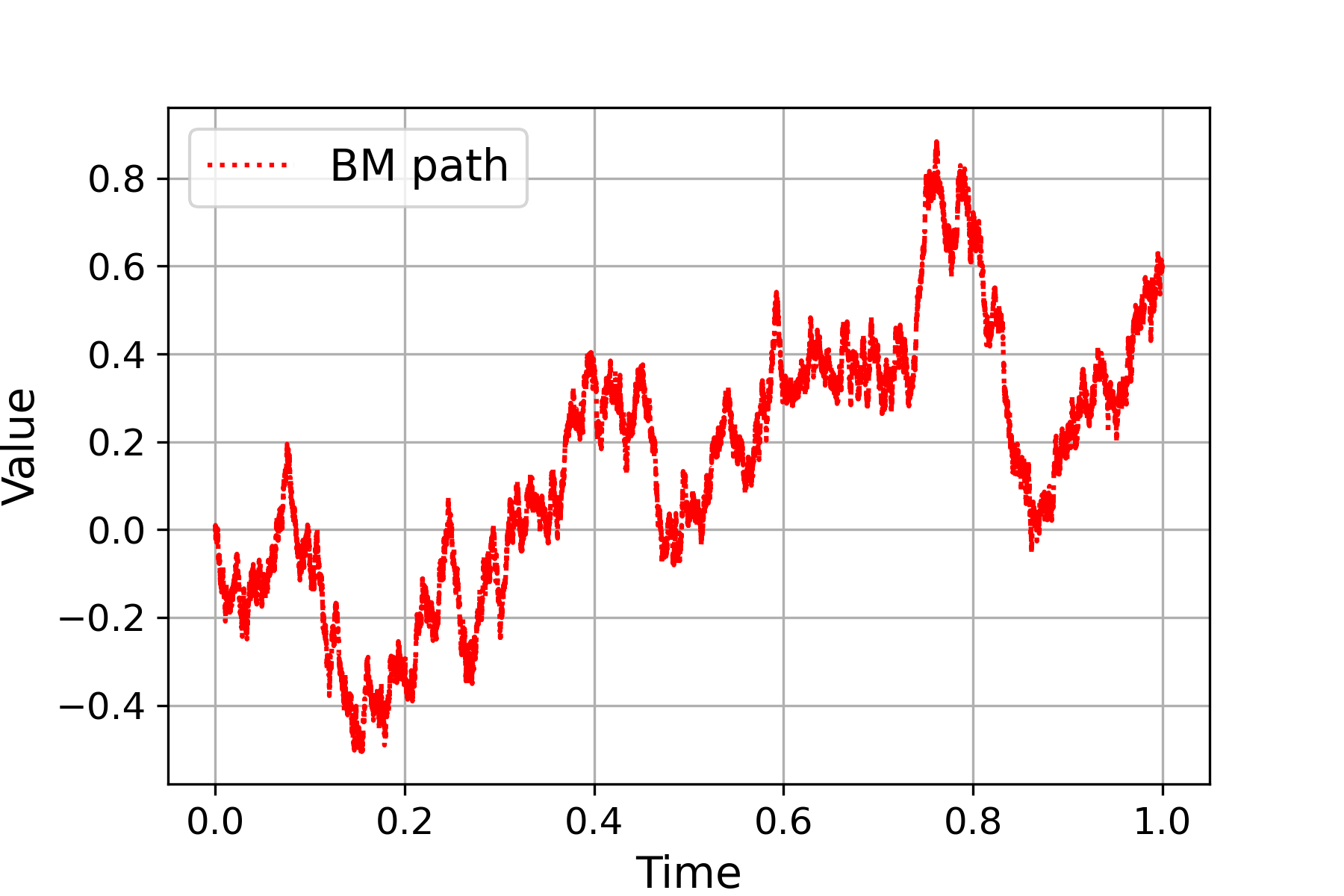}
      \includegraphics[width=0.48\textwidth]{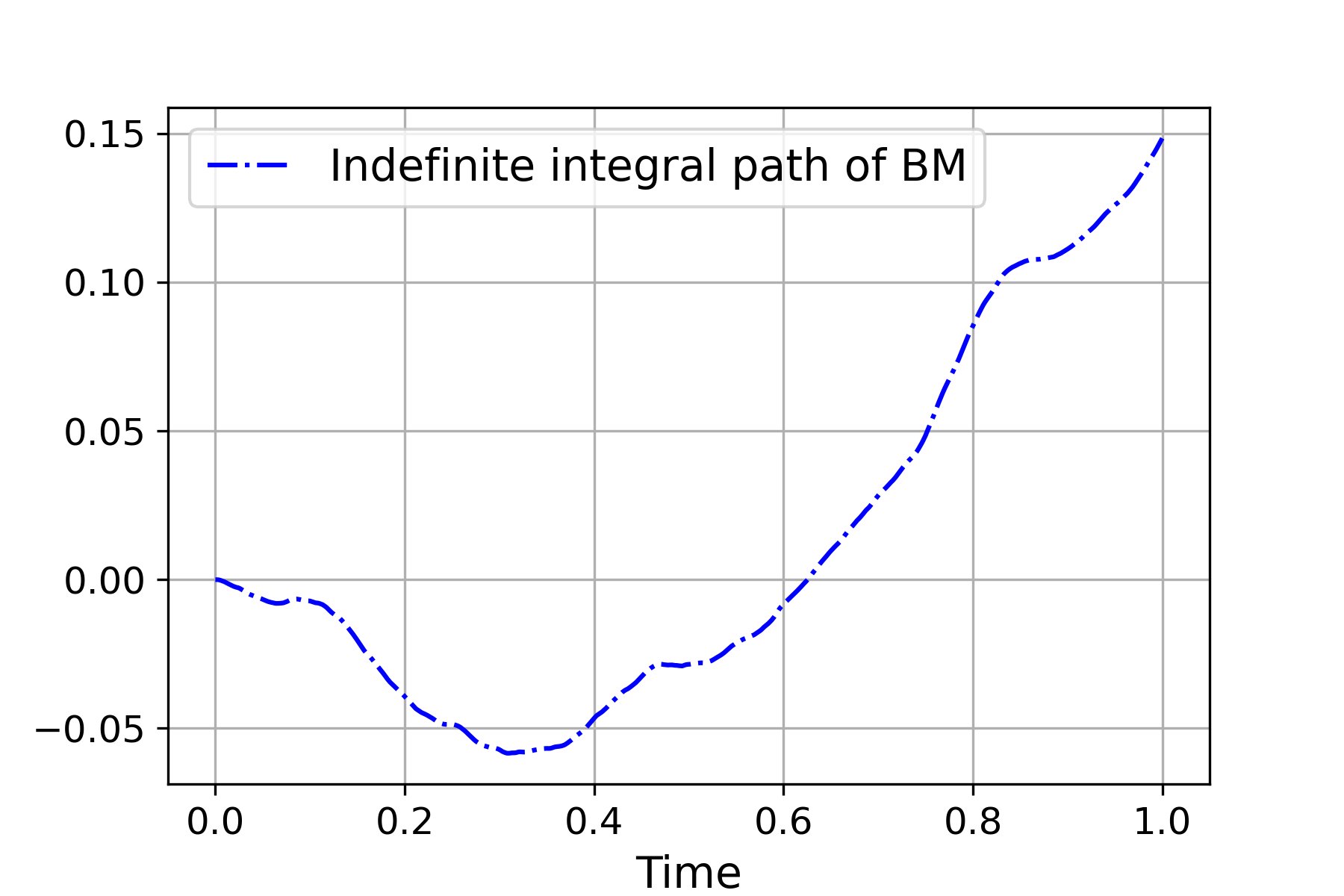}   
  \caption{One realisation of standard Brownian motion and function values for the corresponding $g_B$.  \label{fig:gW_values}}
      \end{figure}
We are interested in approximating $I[g_B]$.

Due to the nature of BM, it is not easy to obtain the exact value of $g_B$. To approximating terms $g_B(t_n)$, we simply apply Euler method, i.e.,
\begin{align*}
    g_B(t_n)=\int_0^{t_n}B(s)\diff{s}=\sum_{i=0}^{n-1}\int_{t_i}^{t_{i+1}}B(s)\diff{s}\approx h\sum_{i=0}^{n-1}B(t_i).
\end{align*}

For CTQ, for a fixed stepsize $h\in [0,1],$ we have that
\begin{align*}
 & Q_h[g_B]= \frac{h}{2}\sum_{n=0}^{N-1}(g_B(t_n)+g_B(t_{n+1}))= h\sum_{n=0}^{N-1}g_B(t_n)+\frac{h}{2}\sum_{n=0}^{N-1} (g_B(t_{n+1})-g_B(t_{n})) \\
 &=h\sum_{n=0}^{N-1}g_B(t_n)+\frac{h}{2}g_B(t_N)=h\sum_{n=1}^{N}g_B(t_n)-\frac{h}{2}g_B(t_N)\approx h^2\sum_{n=0}^{N-1}\sum_{i=0}^{n}B(t_i)-\frac{h^2}{2}\sum_{i=0}^{N-1}B(t_i).
\end{align*}
For RTQ, define the corresponding i.i.d. uniform distributed sequence is $\{\tau^h_j\}_{j\in \N}$, we have a similar expression:
\begin{align}\label{eqn:RTQ_bm}
\begin{split}
 & RQ_h^{\tau,N}[g_B]= \frac{h}{2}\sum_{n=0}^{N-1}(g_B(t_n+\tau^h_{n}h)+g_B(t_{n}+\bar{\tau}^h_n h))\\
 &=h\sum_{n=0}^{N-1}g_B(t_n)+\frac{h}{2}\sum_{n=0}^{N-1}\big(\int_{t_n}^{t_{n}+\tau^h_n h}B(s)\diff{s}+\int_{t_n}^{t_{n}+\bar{\tau}^h_n h}B(s)\diff{s}\big)\\
 &\approx h\sum_{n=0}^{N-1}g_B(t_n)+\frac{h}{2}\sum_{n=0}^{N-1}\Big(\frac{\tau^h_n h}{2} \big(B(t_n)+B(t_n+\tau^h_nh)\big)+\frac{\bar{\tau}^h_nh}{2}  \big(B(t_n)+B(t_n+\bar{\tau}^h_nh)\big)\Big)\\
 &=h\sum_{n=0}^{N-1}g_B(t_n)+\frac{h^2}{4}\sum_{n=0}^{N-1}B(t_n)+\frac{h^2}{4}\sum_{n=0}^{N-1}\big(\tau^h_n B(t_n+\tau^h_nh)+\bar{\tau}^h_n B(t_n+\bar{\tau}^h_nh)\big).
\end{split}
\end{align}
Note that to deduce the third line, we make use of CTQ rather than Euler method. The reason for this is that using Euler method will result in the same expression as $Q_h[g_B]$. It is easy to see the difference between expressions for CTQ and RTQ lies in the last two terms of the equation above.

To compute the reference solution, we first sampled a BM path with a small stepsize $h_{\mbox{ref}}=2^{-14}$. Then we generated an i.i.d. standard uniformly distributed sequence $\{\tau_j\}_{j\in \N}$, and sampled $B((j+\tau_j) h_{\mbox{ref}})$, which is determined by property of Brownian bridge (c.f. Section 3.1 in \cite{2013mc}), i.e.,
\begin{align*}
    B\big((j+\tau_j) h_{\mbox{ref}}\big)\sim \mathcal{N}\Big(\bar{\tau}_jB\big((j+1) h_{\mbox{ref}}\big)+\tau_jB(j h_{\mbox{ref}}),\tau_j\bar{\tau}_jh_{\mbox{ref}}\Big),
\end{align*}
where $\mathcal{N}(\mu,\sigma^2)$ is normal distribution with mean $\mu$ and variance $\sigma^2$, and $\bar{\tau}_j:=1-\tau_j$ for all $j$. The reference solution was thus computed via CTQ on grid points consisting of $\{j h_{\mbox{ref}}\}_{j\in \N}$ as well as these intermediate $\{(j+\tau_j) h_{\mbox{ref}}\}_{j\in \N}$. 

The reason for including randomness at this early stage is that this allows an easier sampling procedure for $\{\tau^h_j\}_{j\in \N}$ on coarser grids of stepsize $h$. For instance, if $h=2h_{\mbox{ref}}$ and consider interval $[t_0,t_0+h]$, then $t_0+\tau_0h_{\mbox{ref}}$ and $t_0+h_{\mbox{ref}}+\tau_1h_{\mbox{ref}}$ are in the same interval. Thus $\tau^h_0$ can be determined from
\begin{align*}
    t_0+\tau^h_0 h=\mathbbm{1}_{U_0}(0) (t_0+\tau_0h_{\mbox{ref}})+\mathbbm{1}_{U_0}(1)(t_0+h_{\mbox{ref}}+\tau_1h_{\mbox{ref}}),
\end{align*}
where $\mathbbm{1}_{\cdot}(\cdot)$ is the indicator function, $U_0\sim \mathcal{U}\{0,1\}$, i.e.,  a discrete uniform distribution on the integers 0 and 1 (shown in Figure \ref{fig:interval}).
  \begin{figure}[h!]
  \centering
 \includegraphics[trim=8cm 3.8cm 8cm 7cm, clip,width=1.0\textwidth]{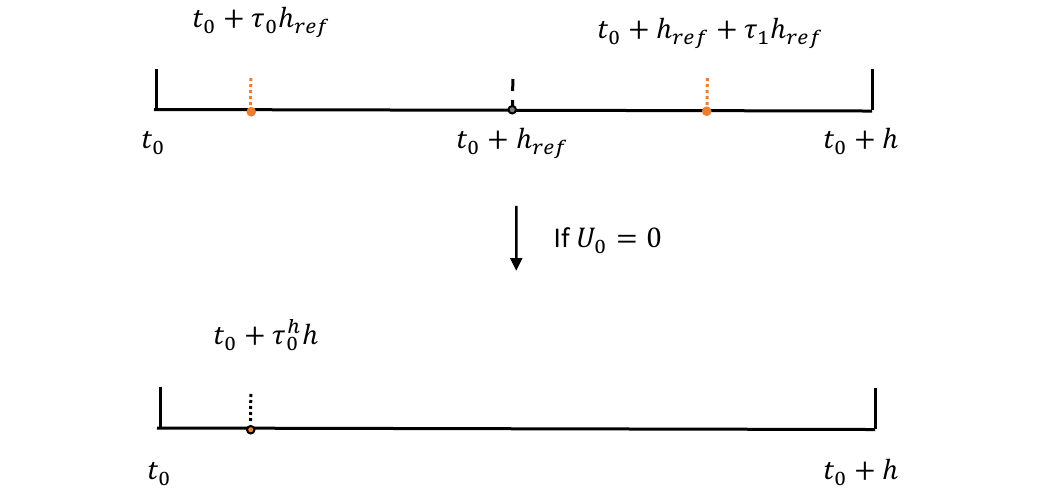}   
  \caption{An illustration of sampling $\tau^h_0$ for coarser grid of stepsize $h$ based on randomness on the finest grid of stepsize $h_{\mbox{ref}}$ under the condition that $h=2h_{\mbox{ref}}$.  \label{fig:interval}}
      \end{figure}

The
numerical approximations were calculated for both trapezoidal quadratures with larger step sizes $h \in \{2^{-i}: i= 5,\ldots, 10\}$ and then compared to the reference solution for errors. The results of our simulations are shown in Figure \ref{fig:gW_error}. RTQ gave the higher order of pathwise convergence compared to CTQ and gained a minor advantage in absolute error. Both the performances are consistent with theoretical order of convergences shown in Theorem \ref{thm:general_trap} and Theorem \ref{th3:Rie_pathwise}. We, in the meantime, examined the computational efficiency of both methods. Due to additional terms involved for RTQ in Eqn. \eqref{eqn:RTQ_bm}, its time cost roughly doubles that of CTQ at the same stepsize. In this case, unfortunately, the slight odds of RTQ in accuracy does not offsets its cost.
  \begin{figure}[h!]

 \includegraphics[width=0.49\textwidth]{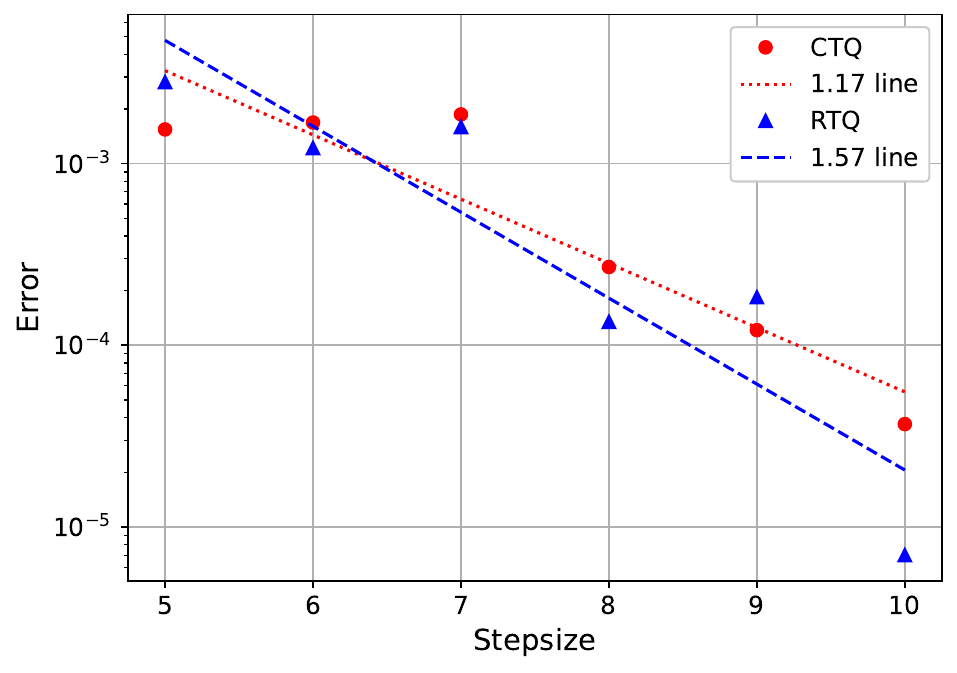}
        \includegraphics[width=0.49\textwidth]{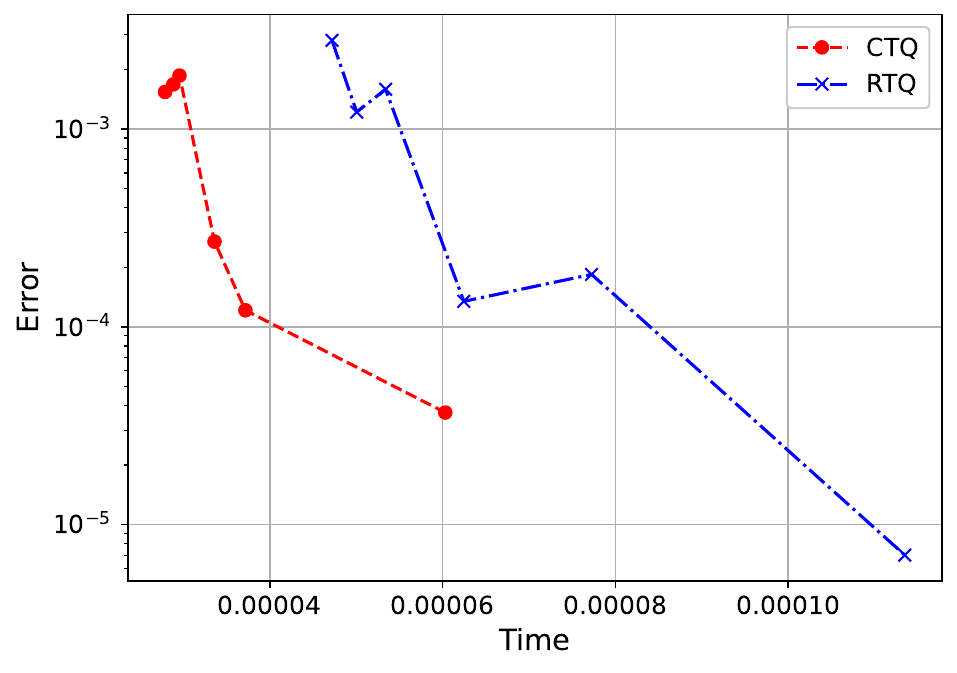} 
  \caption{Error plot (left) and time cost plot (right) for approximating $I[g_B]$ using CTQ and RTQ.  \label{fig:gW_error}}
      \end{figure}

\section*{Acknowledgements}
This work was supported by the The Alan Turing Institute under the EPSRC grant EP/N510129/1 and by EPSRC though the project EP/S026347/1, titled ’Unparameterised
multi-modal data, high order signatures, and the mathematics of data science’.

\end{document}